\documentclass[10pt]{article}
\usepackage{amsmath, amsthm}\usepackage{enumerate}
\usepackage{amssymb}\usepackage{bbold}
\usepackage{color}
\newtheorem{theorem}{Theorem}%[section]

\large
\makeatletter
\renewcommand{\subsection}{\@startsection{subsection}{1}
{0pt}{3.25ex plus 1ex minus.2ex}{-1em}{\normalfont\normalsize\bf}}
\makeatother
\begin{document}

\date{\empty}
\title{{\bf Equi-bounded on order intervals families of semi-norms}}
\maketitle
\author{\centering{{Eduard Emelyanov$^{1}$\\ 
\small $1$ Sobolev Institute of Mathematics, Novosibirsk, Russia}
\abstract{It is proved that equi-bounded on order intervals families of semi-norms on an ordered Banach space with a closed generating cone are equi-continuous.}

\vspace{3mm}

\vspace{3mm}
}}

\section{Introduction}

\hspace{4mm}
It is well known that order bounded operators from a Banach lattice to a normed lattice are continuous 
(see, for example \cite[Theorem 1.31]{AA2002}). In \cite[Theorem 2.1]{E0-2025} it was proved that 
collectively order-to-norm bounded sets of operators from a Banach lattice to a normed lattice are equi-continuous.
The later result was extended in \cite[Theorem 2.1]{EEG2025} to the setting of ordered Banach spaces
with closed generating cones. In the present note, we generalize the above results by proving that 
every equi-bounded on order intervals family of semi-norms on an ordered Banach space with a closed generating cone 
is equi-continuous. 

Throughout, vector spaces are real and ${\cal L}(X,Y)$ stands for the space of linear operators between vector spaces $X$ and $Y$.
The closed unit ball of a normed space $X$ is denoted by $B_X$.
Let $X$ be an ordered vector space. If $Y$ is a normed space, an operator $T\in{\cal L}(X,Y)$ is {\em order-to-norm bounded} whenever $T[a,b]$ is bounded 
for every order interval $[a,b]\subseteq X$. More generally, if $Y$ is a topological vector space,
a family ${\cal T}\subseteq{\cal L}(X,Y)$ is {\em collectively order-to-topology bounded} whenever the 
set ${\cal T}[a,b]=\bigcup\limits_{T\in{\cal T}}T[a,b]$ is bounded for every $[a,b]\subseteq X$. 
A family ${\cal S}$ of semi-norms on a vector space $X$ is {\em equi-bounded on a subset} $A$ of $X$ if ${\cal S}A$ is bounded in $\mathbb{R}$
(in particular, a family ${\cal S}$ of semi-norms on a normed space $X$ is equi-continuous iff it is equi-bounded on $B_X$). 
The set of equi-bounded on order intervals of an ordered vector space $X$ semi-norms is denoted by $\text{\bf S}_{eboi}(X)$. 
For further terminology and notations that are not explained in the text, we refer to \cite{AA2002,AT2007}.

\section{Main result}

\hspace{4mm}
The following theorem extends recent results of \cite[Theorem 2.1]{E0-2025} and \cite[Theorem 2.1]{EEG2025} to the semi-norm setting. 

\begin{theorem}\label{main-1}
If $X$ is an ordered Banach space with a closed generating cone, then every family of semi-norms on $X$ equi-bounded on order intervals is equi-continuous.
\end{theorem}

\begin{proof}
We prove the assertion by contradiction. Assume that ${\cal S}$ is not equi-continuous for some ${\cal S}\in\text{\bf S}_{eboi}(X)$. 
By the Krein -- Smulian theorem (cf., \cite[p.85]{AT2007}), $\alpha B_X\subseteq B_X\cap X_+-B_X\cap X_+$ for some $\alpha>0$,
and hence ${\cal S}$ is not equi-continuous on $B_X\cap X_+$. Then, there exist sequences $(x_n)$ in $B_X\cap X_+$ 
and $(s_n)$ in ${\cal S}$ such that $s_n(x_n)>n^3$ for all $n$. Set $x:=\|\cdot\|$-$\sum\limits_{n=1}^\infty n^{-2}x_n\in X_+$.
Since ${\cal S}\in\text{\bf S}_{eboi}(X)$ then ${\cal S}[0,x]\subseteq[0,N]$ for some $N\in\mathbb{N}$.
It follows from $n^{-2}x_n\in[0,x]$ that $s_n(n^{-2}x_n)\in[0,N]\subseteq[0,n]$ for large enough $n$.
It is absurd, since $s_n(n^{-2}x_n)>n$ for all $n$. 
\end{proof}

\noindent 
Norm completeness of $X$ is essential in Theorem \ref{main-1}. For example, a semi-norm $s$ on $c_{00}$ defined by 
$s(x)=\sum_{k=1}^\infty x_k$ is equi-bounded on order intervals of $c_{00}$, yet not equi-continuous.
The condition that $X_+$ is generating is essential too. To see this, take any semi-norm that is not equi-continuous on a Banach space
$X$ with a trivial cone $X_+=\{0\}$. The closeness of $X_+$ in a Banach space $X$ is also essential.
So see this, take any infinite dimensional Banach space $X$, make it an ordered Banach space, and set a semi-norm $s(x)=\|Tx\|$ as in \cite[Example 2.12 b)]{EEG2025}.

\medskip
The key result of \cite{EEG2025} turns to be a consequence of Theorem \ref{main-1} as follows.

\begin{theorem}\label{cor-1 to main-1}
{\rm (\cite[Theorem 2.1]{EEG2025})}
Let $X$ be an ordered Banach space with a closed generating cone, $Y$ a normed space, and ${\cal T}$ a collectively order-to-norm bounded subset of ${\cal L}(X,Y)$. 
Then ${\cal T}$ is bounded in the operator norm.
\end{theorem}

\begin{proof}
For each $T\in{\cal T}$ we define a semi-norm $s_T$ on $X$ by $s_T(x)=\|Tx\|$.
Since ${\cal T}$ is collectively order-to-norm bounded then ${\cal S}_{\cal T}=\{s_T:T\in{\cal T}\}\in\text{\bf S}_{eboi}(X)$. 
By Theorem \ref{main-1}, ${\cal S}_{\cal T}$ is equi-continuous. So, ${\cal S}_{\cal T}$ is bounded on $B_X$, and hence
${\cal T}$ is norm bounded.
\end{proof}

\medskip
The following consequence of Theorem \ref{main-1} should be compared with \cite[Theorem 2.5]{E2-2025}.

\begin{theorem}\label{cor-2 to main-1}
Let $X$ be an ordered Banach space with a closed generating cone and $(Y,\tau)$ a locally convex topological vector space,
and ${\cal T}$ a collectively order-to-topology bounded subset of ${\cal L}(X,Y)$. Then ${\cal T}$ is equi-continuous.
\end{theorem}

\begin{proof}
Take a locally convex base ${\cal U}\subseteq\tau(0)$ of the topology $\tau$ at zero, denote by $f_U$ the Minkowski functional of $U\in{\cal U}$.
For every $T\in{\cal T}$ and $U\in{\cal U}$ we define a semi-norm $s_T^U$ on $X$ by $s_T^U(x)=|f_U(Tx)|$. 
Since ${\cal T}$ is collectively order-to-topology bounded, the set ${\cal T}[a,b]$ is $\tau$-bounded in $Y$ for every order interval $[a,b]\subseteq X$.
Therefore, ${\cal S}_{\cal T}^U=\{s_T^U:T\in{\cal T}\}\in\text{\bf S}_{eboi}(X)$ for each $U\in{\cal U}$.
Theorem \ref{main-1} implies that the set ${\cal S}_{\cal T}^U$ is bounded on $B_X$ for each $U\in{\cal U}$, say
$\sup\limits_{T\in{\cal T}; x\in B_X}|f_U(Tx)|=\sup\limits_{T\in{\cal T}; x\in B_X}s_T^U(x)\le N_U\in\mathbb{N}$.
It follows ${\cal T}\big(N_U^{-1}B_X\big)\subseteq U$ that every $U\in{\cal U}$. Since ${\cal U}$ is a base of $\tau$ at zero, ${\cal T}$ is equi-continuous.
\end{proof}

\medskip

%\noindent {\bf Declaration of competing interest:}\ None.

%\smallskip \noindent {\bf Funding:}\ The work was carried out in the framework of the State Task to the Sobolev Institute of Mathematics (Project FWNF-2026-0022).

\smallskip
{}
\end{document}